\documentclass{amsart}

\usepackage{amssymb,amsmath,latexsym,amsthm, url}
\usepackage[colorlinks]{hyperref}
\usepackage[english]{babel}

\newtheorem{thm}{Theorem}[section]
\newtheorem{theorem}[thm]{Theorem}
\newtheorem{lemma}[thm]{Lemma}
\newtheorem{corollary}[thm]{Corollary}

\theoremstyle{definition}
\newtheorem{remark}[thm]{Remark}

\numberwithin{equation}{section}

\newcommand{\BB}{\mathbb{B}}

\newcommand{\QQ}{\mathbb{Q}}
\newcommand{\ZZ}{\mathbb{Z}}
\newcommand{\FF}{\mathbb{F}}

\newcommand{\HH}{\mathbb{H}}
\newcommand{\RR}{\mathbb{R}}
\newcommand{\PP}{\mathbb{P}}

\newcommand{\CC}{\mathbb{C}}

\newcommand{\Oc}{\mathcal{O}}

\newcommand{\Lc}{\mathcal{L}}

\newcommand{\isoeq}{\cong}
\newcommand{\cech}{\vee}

\newcommand{\Gal}{\mathrm{Gal}}

\newcommand{\Fil}{\mathrm{Fil}}

\newcommand{\dR}{\mathrm{dR}}

\newcommand{\Aut}{\mathrm{Aut}}

\newcommand{\Spa}{\mathrm{Spa}}

\newcommand{\Hom}{\mathrm{Hom}}

\newcommand{\Lie}{\mathrm{Lie}}

\newcommand{\GL}{\mathrm{GL}}

\newcommand{\un}{\mathrm{un}}

\newcommand{\LT}{\mathrm{LT}}

\newcommand{\HT}{\mathrm{HT}}

\newcommand{\End}{\mathrm{End}}

\newcommand{\GH}{\mathrm{GH}}

\newtheorem{conjecture}{Conjecture}

\begin{document}

\title{Transcendence of the Hodge-Tate filtration}


\author{Sean Howe}
\email{sean.howe@utah.edu}



\thanks{\textit{Acknowledgements.} We thank Jared Weinstein for a helpful conversation about an earlier version of Theorem \ref{theorem:pdivTrans} and possible generalizations. We thank Matt Emerton and an anonymous referee for helpful comments and suggestions.}

\maketitle

\begin{abstract}
For $C$ a complete algebraically closed extension of $\QQ_p$, we show that a one-dimensional $p$-divisible group $G/ \Oc_C$ can be defined over a complete discretely valued subfield $L \subset C$ with Hodge-Tate period ratios contained in $L$ if and only if $G$ has CM, if and only if the period ratios generate an extension of $\QQ_p$ of degree equal to the height of the connected part of $G$. This is a $p$-adic analog of a classical transcendence result of Schneider which states that for $\tau$ in the complex upper half plane, $\tau$ and $j(\tau)$ are simultaneously algebraic over $\QQ$ if and only if $\tau$ is contained in a quadratic extension of $\QQ$. We also briefly discuss a conjectural generalization to shtukas with one paw.
\end{abstract}

\bigskip
\section{Introduction}

\subsection{Transcendence of $\tau$}
An elliptic curve $E$ over the complex numbers $\CC$ equipped with a basis for its integral Betti homology gives rise to a point 
\[ [\tau:1] \in \HH^\pm=\PP^1(\CC) \backslash \PP^1(\RR)\]
describing the position of the Hodge filtration on $H_1(E(\CC), \CC)$ with respect to the fixed basis (thus, $\tau$ is the ratio of two periods of a non-zero holomorphic differential on~$E$). The $GL_2(\ZZ)$-orbits on this set are in bijection with the isomorphism classes of elliptic curves over $\CC$. These isomorphism classes are also parameterized by an algebraic modulus, the $j$-invariant, which has the property that the elliptic curve $E_\tau$ corresponding to $\tau$ has a model over the field $\QQ(j(\tau)).$ A classical transcendence result of Schneider \cite{Schneider-Trans} gives:

\begin{theorem}[Schneider \cite{Schneider-Trans}]\label{thm:Schneider} For $\tau \in \HH^\pm$, the following are equivalent:
\begin{enumerate}
\item $\tau$ and $j(\tau)$ are both in $\overline{\QQ}$.
\item $[\QQ(\tau): \QQ]=2$.
\item $E_\tau$ has CM.
\end{enumerate}
\end{theorem}

\subsection{An analog for $p$-divisible groups}
In this note, we prove a $p$-adic analog of Theorem \ref{thm:Schneider}, where the role of elliptic curves is taken up by one-dimensional $p$-divisible groups, $\CC$ is replaced by a complete algebraically closed $C/\QQ_p$, and $\QQ$ is replaced by any complete discretely valued subfield of $C$ (e.g., $\QQ_p$ or $\widehat{\QQ_p^\un}$, the completion of the maximal unramified extension of $\QQ_p$).

We fix some notation: for a $p$-divisible group $G / \Oc_C$, we denote by $\dim G$ the dimension of $G$ and by $\mathrm{ht} G$ the height of $G$. The Tate module of $G$,
\[ T_p G := \varprojlim G[p^n](C) \] 
is a free $\ZZ_p$-module of rank $\mathrm{ht} G$. We denote by $G^\circ$ the connected part of $G$, and by $\Lie G$ the tangent space to the identity element of $G^\circ$ (thinking of $G^\circ$ as a formal Lie group), which is a free $\Oc_{C}$-module of rank $\dim G$. We denote by $\omega_G$ the cotangent space to the identity element of $G^\circ$; it is dual to $\Lie G$. We denote by $G^\cech$ the dual $p$-divisible group (defined using Cartier duality), which satisfies $\mathrm{ht} G^\cech = \mathrm{ht} G$ and $\dim G^\cech = \mathrm{ht} G - \dim G$. Finally, we let $\ZZ_p(1) := T_p \mu_{p^\infty}$ be the Tate $\ZZ_p$-module, and for any $\ZZ_p$-module $M$ and $n \in \ZZ$, we define Tate twists
\[ M(n) := M\otimes_{\ZZ_p} \ZZ_p(1)^{\otimes n}. \]

After tensoring with $C$, the Tate module of a $p$-divisible group $G$ is equipped with a one-step Hodge-Tate filtration 
\begin{equation}\label{eqn:HTFilt} 0 \rightarrow \Lie G(1)\otimes C \rightarrow T_p G \otimes C \rightarrow \omega_{G^\cech} \otimes C \rightarrow 0. \end{equation}

If $G$ is a one-dimensional $p$-divisible group of height $n$ over $\Oc_C$, $\Lie(G)(1) \otimes C$ is one-dimensional, and, after twisting by $\ZZ_p(-1)$, the Hodge-Tate filtration (\ref{eqn:HTFilt}) gives a point
\[ \HT(G) \in \PP(T_p G \otimes C(-1)). \]
Following Tate \cite{Tate-pdiv}, we can describe $\HT(G)$ explicitly: there is a map
\[ \HT: T_p G^\cech \rightarrow \omega_G \]
given by viewing an element of $T_p G^\cech$ as a map from $G$ to $\mu_{p^\infty}$ and pulling back the invariant differential $\frac{dt}{t}$.  If we fix a basis $e_1,..,e_n$ for $T_p G(-1)$ giving an identification 
\[ \PP(T_p G \otimes C(-1)) \isoeq \PP^{n-1}(C) \]
and a basis $\partial$ for $\Lie G$, we have
\[ \HT(G) = [ (\HT(e_1^*), \partial) : \; (\HT(e_2^*), \partial) : ... :
(\HT(e_n^*), \partial) ]\] 
where $e_i^*$ is the dual basis for $T_p G^\cech$ under the natural duality $T_p G^\cech \isoeq T_p G(-1)^*$. Thus, in this presentation the homogeneous coordinates of $\HT(G)$ are the Hodge-Tate periods of $\partial$. The field of definition, $\QQ_p(\HT(G))$, which depends only on $G$, is generated by the ratios of these periods.

We say that a $p$-divisible group $G / \Oc_C$ has CM if $\End(G) \otimes \QQ_p$ contains a commutative semisimple algebra over $\QQ_p$ of rank equal to $\mathrm{ht}G$. Our main result is

\begin{theorem}\label{theorem:pdivTrans} Let $G / \Oc_C$ be a one-dimensional $p$-divisible group. The following are equivalent:
\begin{enumerate}
\item There is a complete discretely valued field $L \subset C$ such that $G$ can be defined over $\Oc_L$ and $\QQ_p(\HT(G)) \subset L$.
\item $[\QQ_p(\HT(G)) : \QQ_p] =\mathrm{ht} G^\circ $.
\item $G$ has CM.
\end{enumerate} 
\end{theorem}

\begin{remark}
When the conditions of Theorem \ref{theorem:pdivTrans} hold, one can always find a complete discretely valued field $L\subset C$ and a $p$-divisible group $G'/ \Oc_L$ such that $G'_{\Oc_C} \isoeq G$ and $G'$ has CM over $\Oc_L$. It is not typically true, however, that \emph{every} $G' /\Oc_L$ such that $G'_{\Oc_C} \isoeq G$ has CM, even after allowing a finite extension of $L$. For example, over the ring of integers of a finite extension of $\QQ_p$, the only extensions of $\QQ_p / \ZZ_p$ by $\mu_{p^\infty}$ with CM will be those with Serre-Tate coordinate a root of unity, however, after passing to $\Oc_C$ all will have CM, as over $\Oc_C$ the extensions split.
\end{remark}

\begin{remark}
Theorem \ref{theorem:pdivTrans} can be used to produce transcendental elements of $C$ as follows: if $L$ is a complete discretely valued extension of $\QQ_p$ and $G$ is a formal group over $\Oc_L$ without CM, then we find that at least one of the period ratios of $G$ is contained in $\widehat{\overline{L}} \backslash \overline{L}$ (since any algebraic extension of $L$ is also discretely valued). The CM lifts of formal groups are well understood (cf. e.g. \cite{Gross-Canonical} and \cite{Yu-QuasiCan}) -- in particular, using the Gross-Hopkins period map it is easy to see that ``most" formal groups over complete discretely valued extensions of $\QQ_p$ do not have CM. 
\end{remark}

The main tools used in the proof of Theorem \ref{theorem:pdivTrans} are the Scholze-Weinstein \cite{ScholzeWeinstein-modulipdiv}  classification of $p$-divisible groups over $\Oc_C$ by the Hodge-Tate filtration, and Tate's~\cite{Tate-pdiv} theorem on full-faithfulness of the Tate module of a $p$-divisible group over a complete discretely valued field. We recall both results in Section \ref{sec:Recollections}. 

\begin{remark} The Scholze-Weinstein classification identifies the set of isomorphism classes of one-dimensional height~$n$  $p$-divisible groups over $\Oc_C$ with the set of $\GL_n(\ZZ_p)$-orbits in $\PP^{n-1}(C)$. If we restrict to formal groups, then the Hodge-Tate filtration lies in the Drinfeld upper half 
space $\Omega_{n-1}$ (constructed by removing from $\PP^{n-1}$ all lines that are contained in a proper $\QQ_p$-rational subspace). In particular, for $n=2$, we have 
\[ \Omega_1(C) = \PP^1(C) - \PP^1(\QQ_p) \]
which we can write with coordinate $[\tau_\HT: 1].$
Thus, for height $2$ formal groups Theorem \ref{theorem:pdivTrans} becomes a precise analog of Theorem \ref{thm:Schneider} -- the role of the $j$-invariant is played by the explicit statement that $G$ can be defined over $\Oc_L$. 
\end{remark}

\subsection{Period mappings}
No function playing the role of the $j$-invariant appears in Theorem \ref{thm:Schneider} because $\PP^n$ is not a well-behaved moduli space for the $p$-divisible groups we consider. After restricting to formal groups we can remedy this: let $G_0$ be the unique one-dimensional height $n$ formal group over $\overline{\FF_p}$, and let $\LT_n^\infty$ denote the infinite level Lubin-Tate space parameterizing deformations of $G_0$ with a basis of the Tate module, which is a preperfectoid space over $\mathrm{Frac}W(\overline{\FF_p})$ by work of Weinstein \cite{Weinstein-models}. It admits two period maps (cf. \cite{ScholzeWeinstein-modulipdiv}) -- the Hodge-Tate period map
\[ \pi_\HT : \LT_n^\infty \rightarrow \Omega_{n-1}, \]
which is described on $C$-points using the Hodge-Tate filtration (\ref{eqn:HTFilt}) as above, and the Gross-Hopkins \cite{GrossHopkins-equiv} period map
\[ \pi_\GH : \LT_n^\infty \rightarrow \PP^{n-1}. \]
The Gross-Hopkins period map factors through level zero Lubin-Tate space, and remembers the field of definition of a deformation of $G_0$ up to a finite extension. Fixing an embedding $W(\overline{\FF_p}) \rightarrow C$, we obtain

\begin{corollary}\label{cor:PeriodMap}
For $x \in \LT_{n}^\infty(C)$ with corresponding formal group $G_x / \Oc_C$, the following are equivalent:
\begin{itemize}
\item There is a complete discretely valued $L \subset C$ such that $\pi_\GH(x) \in \PP^{n-1}(L)$ and $\pi_\HT(x) \in \Omega_{n-1}(L)$. 
\item $[\QQ_p(\pi_\HT(x)):\QQ_p]=n$.
\item $G_x$ has CM. 
\end{itemize}

\end{corollary}

\subsection{Outline}
In Section \ref{sec:Recollections} we recall some results on $p$-divisible groups. In Section~\ref{sec:Proofs} we prove Theorem \ref{theorem:pdivTrans} and Corollary \ref{cor:PeriodMap}. In Section \ref{sec:Generalizations} we briefly discuss a conjectural generalization to higher dimensional $p$-divisible groups and shtukas with one paw.

\section{Recollections}\label{sec:Recollections}
In this section we recall two theorems on $p$-divisible groups: 

\begin{theorem}[Tate \cite{Tate-pdiv}, Theorem 4]\label{thm:Tate}
For $L$ a complete discretely valued extension of $\QQ_p$, the functor
\[ G \mapsto T_p G \]
is fully faithful from the category of $p$-divisible groups over $\Oc_L$ to the category of finite free $\ZZ_p$-representations of $\Gal(\overline{L}/L)$.   
\end{theorem}

\begin{theorem}[Scholze-Weinstein \cite{ScholzeWeinstein-modulipdiv}, Theorem B]\label{thm:SW} For $C$ a complete algebraically closed extension of $\QQ_p$, the functor 
\[ G \mapsto (T_p G, \Lie G \subset T_p G \otimes C(-1) ) \]
is an equivalence between the category of $p$-divisible groups over $\Oc_C$ and the category of pairs $(M, \Fil)$ consisting of a finite free $\ZZ_p$-module $M$ and a $C$-vector subspace $\Fil \subset M \otimes C(-1)$. 
\end{theorem}

\begin{remark}
We note that for 1-dimensional $p$-divisible formal groups, Theorem \ref{thm:SW} is due to Fargues \cite{Fargues-GroupesAnalytiques}. 
\end{remark}

\section{Proofs}\label{sec:Proofs}

In this section we prove Theorem \ref{theorem:pdivTrans} and then deduce Corollary \ref{cor:PeriodMap} from it. The key observation in the proof of Theorem \ref{theorem:pdivTrans} is that the combination of Tate's theorem (Theorem \ref{thm:Tate} above) and the fact that the Hodge-Tate filtration is determined by the Galois representation put strong restrictions on the field of definition of the Hodge-Tate filtration for a $p$-divisible group over a complete discretely valued field. 

To prove Theorem \ref{theorem:pdivTrans}, we will need a linear algebra lemma. For $F'/F$ an extension of fields, $V$ a finite dimensional vector space over $F$, and $W \subset V \otimes_F F'$ a one-dimensional subspace, we will denote by $F(W)$ the field of definition of $W$ inside of $F'$. It is the smallest subfield $L \subset F'$ such that there exists a subspace $W_L \subset V \otimes_F L$ with $W_L \otimes_L F = W$. Equivalently, it is the field generated over $F$ by the ratios of the homogeneous coordinates of $W$ as a point in $\PP(V \otimes F')$ with respect to a fixed $F$-basis of $V$. 

\begin{lemma}\label{lem:Stabilizer} Let $V$ be a finite dimensional vector space over a perfect field $F$, and let $F'/F$ be an extension containing an algebraic closure of $F$. If 
\[ W \subset V \otimes F' \]
is a line whose conjugates under $\Aut(F'/F)$ span $V \otimes F'$, then the stabilizer of $W$ in $\End_{F}(V)$ is isomorphic to a subfield $K \subset F(W)$ such that $[K:F] \leq \dim V$.  Furthermore, in this situation $K=F(W)$ if and only if $[F(W):F]=\dim V$ if and only if ${[K:F]=\dim V}$. 
\end{lemma}
\begin{proof}
We denote $d=[F(W):F]$ (which could be $\infty$), and $n=\dim V$. We denote by $K$ the stabilizer of $W$ in $\End_F(V)$. 

The action of $K$ on the line 
\[ W_{F(W)} := W \cap V\otimes_F F(W) \]
induces a map from $K$ to $F(W)$. It is an injection because any element of $\End_{F}(V)$ that acts as zero on $W$ also acts as zero on all of its conjugates under $\Aut(F'/F)$, which span $V \otimes_F F'$. Thus we obtain $K \hookrightarrow F(W)$. In particular, $K$ is a field. Then $V$ is a $K$-vector space, so we find $[K:F]\leq n$.  

We now show $[K:F]=n$ if and only if $K=F(W)$ if and only if $[F(W):F]=n$. 

Suppose $[K:F]=n$. Then, $V$ is isomorphic to $K$ as a $K$-module. Because $F'/F$ contains $\overline{F}$, and $K/F$ is separable ($F$ is perfect), we find $V \otimes F'$ splits into $n$ distinct characters of $K$, and thus these are the only lines stabilized by $K$. Each of these lines is defined over an extension of degree $n$ of $F$, so that $[F(W):F]\leq n$. Since $K \subset F(W)$, $[F(W):F]=n$, and $K=F(W)$.  

Suppose $K=F(W)$. Because $W$ has $n$ distinct conjugates, we have $d \geq n$. Because $[K:F] \leq n$, we have $d \leq n$. Thus $d=n$, and $[K:F]=[F(W):F]=n$.  

Suppose $[F(W):F]=n$. It suffices to show $F(W) \hookrightarrow K$, since that implies $[K:F] \geq n$, and we obtain $[K:F]=n$ and $K=F(W)$.  We observe 
\[ V^* \isoeq (V^* \otimes F(W)) / W_{F(W)}^\perp \]
as $F$ vector spaces, since both have dimension $n$, and the map is injective (if it were not, there would be a non-zero linear form on $V$ defined over $F$ and vanishing on $W$, thus the conjugates of $W$ would not span $V \otimes F'$). The right hand side is a $F(W)$-vector space, so this isomorphism equips $V^*$ with an action of $F(W)$ preserving $W^\perp \subset V^* \otimes F'$, and thus dually equips $V$ with an action of $F(W)$ preserving $W \subset V \otimes F'$. This gives the desired inclusion $F(W) \hookrightarrow K$. 
\end{proof}

\begin{remark}\label{rem:HigherDim}
We will use Lemma \ref{lem:Stabilizer} in the proof of Theorem \ref{theorem:pdivTrans} to force the image of the Galois representation on the Tate module of a formal group to be abelian. For higher dimensional formal groups this step breaks down because for a higher dimensional filtration the stabilizers in Lemma \ref{lem:Stabilizer} are no longer necessarily commutative.  
\end{remark}

We will also need the following structural lemma:

\begin{lemma}\label{lem:LieSpanConn} If $G/\Oc_C$ is a $p$-divisible group, then 
\[ G \isoeq G^\circ \times (\QQ_p/\ZZ_p)^{\mathrm{ht} G - \dim G}, \] and the conjugates of $\Lie G$ under $\Aut(C / \QQ_p)$ span $T_p G^\circ(-1) \otimes C$.
\end{lemma} 
\begin{proof}
Let $W'$ be the span of the conjugates of ${W=\Lie G}$ in $T_pG(-1) \otimes C$, and let $W'_0 = (W' \cap T_p G(-1)) (1)$, which we identify with a submodule of $T_p G$. $W'_0$ is a free $\ZZ_p$-module of rank $\dim W'$, saturated in $T_p G$ (here we use some standard descent results for vector spaces to deduce that $W' \cap T_p G \otimes \QQ_p$ has $\QQ_p$-dimension equal to $\dim W'$ and spans $W'$ -- cf., e.g.,\cite[Propositions 16.1 and 16.7]{Milne-AG}).

By the Scholze-Weinstein classification, Theorem \ref{thm:SW}, the pair $(W'_0, \Lie G)$ defines a $p$-divisible group $H$ over $\Oc_C$. Furthermore, any map from $H$ to $\QQ_p/\ZZ_p$ comes from a map from $(W'_0, \Lie G)$ to $(\ZZ_p, \{0\})$, and thus must be zero since it sends $\Lie G$ and all of its conjugates to zero. Thus, $H$ is connected. If we choose a free $\ZZ_p$-module $M$ of rank $m$ complementary to $W$ in $T_p G$ (which exists because $W'_0$ is saturated in $T_p G$), and a trivialization $\ZZ_p^m \isoeq M$, then the resulting isomorphism
\[ (W'_0, \Lie G) \times (\ZZ_p^m, \{0\}) \isoeq (T_p G, \Lie G) \]
gives an isomorphism 
\begin{equation}\label{eqn:prelimDecomp} H \times (\QQ_p/\ZZ_p)^m \rightarrow G. \end{equation}
Since $H$ is connected, we deduce $H \isoeq G^\circ$, and $m = \dim G - \mathrm{ht} G$. Thus (\ref{eqn:prelimDecomp}) is the desired decomposition, and the statement about the conjugates of $\Lie G$ follows from the construction of $W'$.  
\end{proof}

\begin{proof}[Proof of Theorem \ref{theorem:pdivTrans}]
We first reduce to the case of $G$ connected (i.e. to $G$ a formal group). Observe that passing to the connected component $G^\circ$ does not change $\QQ_p(\HT(G))$, so that condition $(2)$ of the theorem is unchanged. Furthermore, if $G$ can be defined over $\Oc_L$ as in $(1)$, then so can $G^\circ$, and vice versa by Lemma \ref{lem:LieSpanConn}. Thus, if we have $(1) \iff (2)$ for $G$ connected, we obtain $(1) \iff (2)$ for all $G$. Furthermore, the decomposition of Lemma \ref{lem:LieSpanConn} and the Scholze-Weinstein classfication (Theorem \ref{thm:SW}) imply that 
\[  \End(G) = \begin{pmatrix} \End(G^\circ) & \Hom( (\QQ_p/\ZZ_p)^{n-\mathrm{ht}G^\circ}, G^\circ ) \\ 0 & M_{n- \mathrm{ht}G^\circ} (\QQ_p) \end{pmatrix}. \]
Here, the lower left is 0 because any map from a connected finite group scheme to an \'{e}tale finite group scheme is zero (alternatively, this can be seen from the Scholze-Weinstein classification). From this we see that $(2) \iff (3)$ for $G$ connected implies $(2) \iff (3)$ for all $G$. 

From now on we will assume $G$ is connected. Denote by $W$ the Hodge-Tate filtration in $T_p G(-1) \otimes C$. By Lemma \ref{lem:LieSpanConn}, the conjugates of $W$ span $T_p G(-1)\otimes C$. Note $\QQ_p(W)$ is the field $\QQ_p(\HT(G))$ in the statement of the theorem. Now, as a consequence of the Scholze-Weinstein classification (Theorem \ref{thm:SW}), $\End(G) \otimes \QQ_p$ is equal to the stabilizer of $W$ in $\End(T_pG \otimes \QQ_p)$. We then obtain $(2) \iff (3)$ from Lemma \ref{lem:Stabilizer}. In general, we denote this stabilizer, which is equal to $\End G \otimes \QQ_p$, by $K$. By Lemma \ref{lem:Stabilizer}, it is a field with $[K:\QQ_p]\leq n$. 

We now show $(1) \implies (2)$. We assume $(1)$, and by abuse of notation write $G/\Oc_L$ for some $p$-divisible group over $\Oc_L$ with base change to $\Oc_C$ equal to $G$. Because the induced $L$-semilinear action of $\Gal(\overline{L}/L)$ on $T_p G(-1) \otimes L$ preserves the Hodge-Tate filtration and is in fact $L$-linear, the line $W$ is preserved by the linear action of $\Gal(\overline{L}/L)$. On the other hand, the image of $\Gal(\overline{L}/L)$ is also contained in $\End_{\QQ_p}(T_p G(-1) \otimes \QQ_p)$. Thus, Galois acts through $K$, and Tate's theorem (Theorem \ref{thm:Tate}) implies that the centralizer of $K$ also acts by isogenies. So, we must have that $K$ contains its centralizer in $\End_{\QQ_p}(T_p G(-1) \otimes \QQ_p)$, which implies that $[K:\QQ_p]=n$. Thus, by Lemma \ref{lem:Stabilizer}, $[\QQ_p(\HT(G)):\QQ_p]=n.$

Finally we show $(3) \implies (1)$. For a fixed $[K:\QQ_p]=n$, our computations so far along with the Scholze-Weinstein classification (Theorem \ref{thm:SW}) show that isogeny classes of $G$ with CM by $K$ correspond to orbits of $GL_n(\QQ_p)$ on $\Omega_{n-1}(K)$. There is a unique such orbit, corresponding to bases for $K/\QQ_p$ up to $\QQ_p^\times$ homothety. Any Lubin-Tate formal group $G/\Oc_K$ for $K$ is contained in this isogeny class and any group isogenous to it is defined over the ring of a integers of a finite extension of $K$, thus we conclude any formal group with CM by $K$ can in fact be defined over the ring of integers of a finite extension of $\QQ_p$. 

\end{proof}

\begin{proof}[Proof of Corollary \ref{cor:PeriodMap}]
For a point $x \in \LT_n^\infty$, we will show (1) of Corollary \ref{cor:PeriodMap} is equivalent to (1) of Theorem \ref{theorem:pdivTrans} for $G_x$. It suffices to show that if $\pi_\GH(x)$ is in a discretely valued extension $L$ of $\mathrm{Frac}W(\overline{k})$ then $G_x$ can be defined over a finite extension of $L$. As in \cite{GrossHopkins-equiv}, Corollary 23.21, $\pi_\GH$ is surjective (at level 0) on $\overline{L}$ points. The fibers of $\pi_{\GH}$ contain isogenous groups, and thus there is some $G'$ isogenous to $G_x$ defined over a finite extension of $L$. But then $G_x$ is also defined over a finite extension of $L$, since the kernel of an isogeny from $G'$ to $G_x$ is defined over a finite extension.  
\end{proof}

\section{Generalizations}\label{sec:Generalizations}

In the archimedean setting, an analog of Theorem \ref{thm:Schneider} holds for abelian varieties of higher dimension (cf. \cite{Cohen-Gen} and \cite{ShigaWolfart-Gen}). Our method of proof in the $p$-adic setting does not obviously generalize to $p$-divisible groups of dimension~${\geq 1}$ (cf. Remark \ref{rem:HigherDim}), however, it is natural to conjecture the analogous result.

We might even go further: $p$-divisible groups give the simplest examples of shtukas with one paw over $\Spa C^\flat$ as in \cite{Weinstein-ScholzeLectures}. By a result of Fargues (cf. \cite[Theorem 12.4.4]{Weinstein-ScholzeLectures}), a shtuka with one paw is equivalent to a pair consisting of a finite free $\ZZ_p$-module $M$ and a $\BB^+_\dR$-lattice $\Lc \subset M \otimes \BB_\dR$. The lattice $\Lc$ induces a filtration on $M \otimes \BB_{\dR}$,  thus by restriction on $M \otimes \BB^+_\dR$, and then, via specialization along the natural map $\theta : \BB^+_\dR \rightarrow C$, on $M \otimes C$. The resulting filtration on $M \otimes C$ is called the Hodge-Tate filtration. This generalizes the Hodge-Tate filtration on a $p$-divisible group, in which case $M=T_p G$ and $\Lc$ is uniquely determined by the Hodge-Tate filtration and the requirement that it lie in the natural minuscule Schubert cell relative to $M \otimes \BB^+_\dR$. We note that, in general, the Hodge-Tate filtration can have multiple steps! 

We say a shtuka $(M, \Lc)$ has CM if $(M\otimes\QQ_p, \Lc)$ admits endomorphisms by a semi-simple commutative algebra over $\QQ_p$ of dimension equal to the rank of $M$. 

There is also a natural analog of being defined over a complete discretely valued subfield in this setting: 
Given a complete discretely valued field $L \subset C$ and a lattice $M$ in a deRham representation of $\Gal(\overline{L}/L)$, we obtain a shtuka with one paw from the pair
\[(M, (M \otimes \BB_\dR)^{\Gal(\overline{L}/L)} \otimes_{L} \BB^+_\dR ). \]
We will say a shtuka with one paw is \emph{arithmetic} if it is isomorphic to one of this form. We formulate the optimistic

\begin{conjecture}
If $(M, \Lc)$ is an arithmetic shtuka with one paw with Hodge-Tate filtration defined over a complete discretely valued subfield of $C$ then $(M,\Lc)$ has CM.
\end{conjecture}

\begin{remark}
For a CM shtuka, the Hodge-Tate filtration is algebraic since any sub-space preserved by the CM algebra $K$ is a direct sum of the $1$-dimensional character spaces for $K$. 
\end{remark}


\bigskip 


\begin{thebibliography}{xx}

\bibitem{Cohen-Gen} \textsc{Paula Beazley Cohen}, \textit{Humbert surfaces and transcendence properties of automorphic functions}. Rocky Mountain J. Math. {\bf26} (1996), 987--1001.

\bibitem{Fargues-GroupesAnalytiques} \textsc{Laurent Fargues}, \textit{Groupes analytiques rigides $p$-divisibles}. Math. Ann., to appear. 

\bibitem{Gross-Canonical} \textsc{Benedict H. Gross}, \textit{On canonical and quasicanonical liftings}. Invent. Math. {\bf84} (1986), 321--326.

\bibitem{GrossHopkins-equiv} \textsc{M. J. Hopkins and B. H. Gross}, \textit{Equivariant vector bundles on the {L}ubin-{T}ate moduli space}. Topology and representation theory ({E}vanston, {IL}, 1992). Amer. Math. Soc., Providence, RI (1994), 23--88.

\bibitem{Milne-AG} \textsc{J. S. Milne}, \textit{Algebraic {G}eometry}. Available at: \url{http://www.jmilne.org/math/CourseNotes/ag.html}

\bibitem{Schneider-Trans} \textsc{Theodor Schneider}, \textit{Arithmetische {U}ntersuchungen elliptischer {I}ntegrale}. Math. Ann. {\bf113} (1937), 1--13.

\bibitem{ScholzeWeinstein-modulipdiv}\textsc{Peter Scholze and Jared Weinstein}, \textit{Moduli of {$p$}-divisible groups}. Camb. J. Math. {\bf1} (2013), 145--237.

\bibitem{ShigaWolfart-Gen} \textsc{Hironori Shiga and J{\"u}rgen Wolfart}, \textit{Criteria for complex multiplication and transcendence properties of automorphic functions}. J. Reine Angew. Math. {\bf463} (1995), 1--25.


\bibitem{Tate-pdiv} \textsc{J. T. Tate}, \textit{{$p$}-divisible groups}. Proc. {C}onf. {L}ocal {F}ields ({D}riebergen, 1966). Springer, Berlin (1967), 158--183.

\bibitem{Weinstein-ScholzeLectures} \textsc{Jared Weinstein}, \textit{Peter {Scholze's} lectures on $p$-adic geometry}. Available at: \url{http://math.bu.edu/people/jsweinst/Math274/
Lectures.pdf}.

\bibitem{Weinstein-models} \textsc{Jared Weinstein}, \textit{Semistable models for modular curves of arbitrary level}. Invent. Math. {\bf205} (2016), 459-526.

\bibitem{Yu-QuasiCan} \textsc{Jiu-Kang Yu}, \textit{On the moduli of quasi-canonical liftings}. Compositio Math. {\bf96} (1995), 293--321.

\end{thebibliography}
\end{document}